\documentclass{amsart}
\pdfoutput=1

\usepackage[utf8]{inputenc}
\usepackage{url}
\usepackage{doi}
\usepackage{hyperref}

\usepackage{amsmath}
\usepackage{amssymb}
\usepackage{amsthm}
\usepackage{amsfonts}
\usepackage{colonequals}

\providecommand{\op}[1]{\operatorname{#1}}
\newcommand{\wc}{{\mkern 2mu\cdot\mkern 2mu}}
\newcommand{\R}{\mathbb{R}}
\newcommand{\N}{\mathbb{N}}
\providecommand{\st}{\mathrel\vert}
\providecommand{\Angle}{\measuredangle}

\newcommand{\vertiii}{{\vert\kern-0.25ex\vert\kern-0.25ex\vert}}
\newcommand{\triplenorm}[1]{{\vertiii #1 \vertiii}}
\newcommand{\Triplenorm}[1]{{\left\vert\kern-0.25ex\left\vert\kern-0.25ex\left\vert #1 
    \right\vert\kern-0.25ex\right\vert\kern-0.25ex\right\vert}}

\newcommand{\LAiii}{{\langle\kern-0.4ex\langle\kern-0.4ex\langle}}
\newcommand{\RAiii}{{\rangle\kern-0.4ex\rangle\kern-0.4ex\rangle}}
\newcommand{\tripleprod}[1]{\LAiii #1 \RAiii}

\newcommand{\Prod}[1]{\langle #1 \rangle}
\DeclareMathOperator{\Textint}{{\textstyle\int}}

\newtheorem{theorem}{Theorem}
\newtheorem{lemma}{Lemma}
\newtheorem{corollary}{Corollary}
\newtheorem{proposition}{Proposition}

\newtheorem{example}{Example}

\newtheorem{remark}{Remark}

\begin{document}
\title{A note on Poincar\'e- and Friedrichs-type inequalities}

\author[C. Gräser]{Carsten Gräser}
\address{Carsten Gräser\\
Friedrich-Alexander-Universität Erlangen--Nürnberg\\
Department Mathematik\\
Cauerstr. 11\\
91058 Erlangen, Germany\\
 \href{https://orcid.org/0000-0003-4855-8655}{orcid.org/0000-0003-4855-8655}
}
\email{graeser@math.fau.de}

\thanks{We thank H. Garcke for pointing out an error for
an intermediate result on Navier boundary conditions in $H^2$
in a previous version of this manuscript.}

\keywords{Poincare inequality, coercivity, plate problem}
\maketitle

\begin{abstract}
We introduce a simple criterion to check coercivity
of bilinear forms on subspaces of Hilbert-spaces and Banach-spaces.
The presented criterion allows to derive many standard
and non-standard variants of
Poincar\'e- and Friedrichs-type inequalities with very
little effort.
\end{abstract}

\section{Introduction}
Poincar\'e- and Friedrichs-type inequalities play an important role
in existence theory for elliptic and parabolic partial differential
equations because they allow to show coercivity of bilinear forms
on subspaces of Sobolev spaces. In many applications those
bilinear forms are obtained from the natural inner product of
the Sobolev space by incorporating non-constant coefficients,
dropping lower order derivatives, or adding modified lower order
terms. The considered subspaces are obtained by imposing boundary
or other conditions on solutions.
As a consequence a variety of Poincar\'e- and Friedrichs-type inequalities
where proposed to deal with different bilinear forms or constraining conditions.
Each of these is often proved independently.

The aim of this paper is not to show a specific new variant
of such an inequality. Instead we give simple critera to check coercivity
of bilinear forms, which allows to link many variants of 
Poincar\'e- and Friedrichs-type inequalities. The purpose of this
is two-fold: On the one hand it allows to avoid time consuming research
for a published suitable variant in non-standard situations. On the
other hand it can be used in teaching to easily derive the most
common variants with little effort.

The main critera are introduced in Section~\ref{sec:coercivity}
in a general Hilbert-space setting.
A related result in Banach spaces is also given.
In Section~\ref{sec:poincare}
we show how many variants of Poincar\'e- and Friedrichs-type
inequalities can be derived from a single one using this criterion.
Examples incorporate the most common, as well as some non-standard
variants.
Finally, we apply this in Section~\ref{sec:biharmonic} to
derive coercivity for special boundary conditions of forth-
and eighth order problems.

\section{Coercivity on subspaces of Hilbert- and Banach-spaces}
\label{sec:coercivity}

In the following we will call a bilinear form $a(\wc,\wc)$ coercive
on a normed space $V$ with constant $\gamma>0$ if
\begin{align*}
    \gamma \|v\|^2 \leq a(v,v) \qquad \forall v \in V.
\end{align*} 
First we show an auxiliary result linking the angle between
subspaces to norms of orthogonal projections.

\begin{lemma}\label{lemma:angle}
    Let $H$ be a Hilbert-space, $V,W\subset H$ closed subspaces of $H$ with
    $\op{dim}(W)<\infty$ and $V\cap W = \{0\}$.
    Then we have
    \begin{align}
        0 \leq \sup_{v \in V\setminus\{0\}, w \in W\setminus\{0\}} \frac{(v,w)}{\|v\| \|w\|} = \alpha(V,W) <1,
    \end{align} 
    i.e., $\Angle(V,W) = \op{acos}(\alpha(V,W)) >0$.
\end{lemma} 
\begin{proof}
    Assume that this is not the case, then there are sequences $v_n \in V$ and $w_n \in W$
    with $\|v_n\|=\|w_n\|=1$ for all $n$ and $(v_n,w_n) \to 1$.
    By compactness there are subsequences (wlog. also denoted by $v_n$ and $w_n$) and $v \in V$,
    $w\in W$ with $v_n \rightharpoonup v$ and $w_n \to w$. Then we have
    \begin{align*}
        (v_n,w_n) \to (v,w)=1,
    \end{align*}
    $\|w\|=1$, and furthermore by Hahn--Banachs theorem $\|v\| \leq 1$.
    As a consequence we get $\|v-w\|^2 = \|v\|^2 - 2(v,w) + \|w\|^2 \leq 0$ and thus
    $0\neq v=w \in V \cap W$ which contradicts the assumption.
\end{proof} 

Next we use this result to prove an inverse estimate on orthogonal projections
into suitable subspaces.

\begin{lemma}\label{lemma:inverse_ineq}
    Let $H$ be a Hilbert-space, $V,W\subset H$ closed subspaces of $H$ with
    $\op{dim}(W)<\infty$ and $V\cap W = \{0\}$.
    Then the orthogonal projections $P:H \to W$ and $(I-P):H \to W^\perp$
    satisfy the inequalities
    \begin{align*}
        \|P v\| \leq \alpha(V,W) \|v\|, \qquad
        \|v\| \leq \frac{1}{\beta(V,W)} \|(I-P)v\|\qquad \forall v \in V
    \end{align*}
    with constants $\alpha(V,W) = \cos(\Angle(V,W))<1$ and $\beta(V,W) = \sin(\Angle(V,W))>0$.
\end{lemma} 
\begin{proof}
    By Lemma~\ref{lemma:angle} we have $\alpha(V,W) = \cos(\Angle(V,W))<1$
    and thus $\beta(V,W) = \sqrt{1-\alpha(V,W)^2}>0$.
    Now let $v \in V$. Then we have
    \begin{align*}
        \|Pv\|^2
            = (Pv,v-(I-P)v)
            = (Pv,v)
            \leq \alpha(V,W) \|Pv\| \|v\|
    \end{align*} 
    and thus $\|Pv\| \leq \alpha(V,W) \|v\|$.
    Hence we get
    \begin{align*}
        \|v\|^2  = \|(I-P)v\|^2 + \|Pv\|^2
        \leq \|(I-P)v\|^2 + \alpha(V,W)^2 \|v\|^2.
    \end{align*}
    Subtracting $\alpha(V,W)^2 \|v\|^2$ and taking the square root provides the assertion.
\end{proof} 

Using these results we are now ready to show a general criterion for coercivity
on subspaces.
To this end we note that the (left-)kernel of a bilinear form
$a(\wc,\wc) : H \times H \to \R$ is defined by
\begin{align*}
  \op{ker} a = \{v \in H \st a(v,w)=0 \, \forall w \in H\}
\end{align*}
and that the bilnear form if called \textit{reflexive} if
\begin{align*}
  \Bigl(a(v,w) = 0 \quad \Rightarrow \quad a(w,v) = 0\Bigr) \qquad \forall v,w \in H.
\end{align*}
If $a(\wc,\wc)$ is symmetric it is also reflexive and if $a(\wc,\wc)$ is symmetric
and positive semi-definite, then the Cauchy-Schwarz-inequality implies that
the kernel can equivalently be characterized as
\begin{align*}
  \op{ker} a = \{v \in H \st a(v,v)=0\}.
\end{align*}

\begin{proposition}\label{prop:coercivity}
    Let $H$ be a Hilbert-space and $a(\wc,\wc) : H \times H \to \R$
    a reflexive (e.g. symmetric) bilinear form
    with finite dimensional kernel $\op{ker} a$.
    Furthermore assume that
    $a(\wc,\wc)$ is coercive on $(\op{ker}a)^\perp$ with constant $\gamma > 0$.
    Then $a(\wc,\wc)$ is coercive with constant $\gamma \beta(V,\op{ker} a)^2 > 0$
    on any closed subspace $V$
    of $H$ with $V\cap \op{ker} a = \{0\}$.
\end{proposition}
\begin{proof}
    Let $v \in V$ and $P:H \to \op{ker} a$ the orthogonal projection
    into $\op{ker} a$.
    Then Lemma~\ref{lemma:inverse_ineq}
    and coercivity on $(\op{ker}a)^\perp = (I-P)(H)$
    provide $\gamma \beta(V,\op{ker}a)^2>0$ and 
    \begin{align*}
        \begin{split}
            \gamma\beta(V,\op{ker} a)^2 \|v\|^2
            \leq \gamma \|(I-P)v\|^2
            &\leq a((I-P)v,(I-P)v)
                = a(v,v).
        \end{split} 
    \end{align*} 
\end{proof} 

In many situation coercivity is not obtained by restriction to
suitable subspaces, but by augmenting the bilinear form $a(\wc,\wc)$
in order to obtain coercivity on the whole space.

\begin{proposition}\label{prop:coercivity_augmented}
    In addition to the assumptions of Proposition~\ref{prop:coercivity}
    assume that $b(\wc,\wc) :H \times H \to \R$ is
    a continuous positive semi-definite bilinear form which is
    positive definite on $\op{ker} a$.
    Then $a(\wc,\wc) + b(\wc,\wc)$ is coercive on $H$.
\end{proposition}
\begin{proof}
    First we note that Young's inequality (applied to the symmetric part of $b(\wc,\wc)$) provides
    \begin{align*}
      b(x,x) \leq 2 b(y,y) + 2 b(x-y,x-y) \qquad \forall x,y \in H.
    \end{align*}
    Since $\op{ker} a$ is finite dimensional, positive definiteness implies
    coercivity of $b(\wc,\wc)$ on $\op{ker} a$ with some constant $\delta>0$.
    For $P$ as in the proof of Proposition~\ref{prop:coercivity}
    we now have
    \begin{align*}
        \begin{split}
            \|v\|^2
                &= \|v-Pv\|^2 + \|Pv\|^2
                \leq \|v-Pv\|^2 + \delta^{-1} b(Pv,Pv) \\
                &\leq \|v-Pv\|^2  + 2 \delta^{-1} \Bigl( b(v,v) + b(v-Pv,v-Pv) \Bigr) \\
                &\leq (1+2\delta^{-1}C)\|v-Pv\|^2 + 2\delta^{-1} b(v,v) \\
                &\leq (1+2\delta^{-1}C)\gamma^{-1} a(v,v) + 2 \delta^{-1} b(v,v).
        \end{split}
    \end{align*}
\end{proof}

\begin{remark}
  The finite dimensionality required in Lemma~\ref{lemma:angle}, Lemma~\ref{lemma:inverse_ineq},
  and Proposition~\ref{prop:coercivity} is ineed necessary for the respective statements.
  To see this we give counter examples based on the same setting:
  Consider the sequence space $H=\ell^2$ with the canonical orthonormal
  basis $(e^n)_{n \in \mathbb{N}}$ where $e^n_k = \delta_{kn}$
  and define the sequences $(v^n)$ and $(w^n)$ by
  \begin{align*}
    v^n &= e^{2n} + \frac{1}{n}e^{2n+1}, &
    w^n &= e^{2n}.
  \end{align*}
  Then the closed spanned subspaces
  \begin{align*}
    V &= \overline{\operatorname{span}\{v^n\,|\, n \in \mathbb{N}\}}, &
    W &= \overline{\operatorname{span}\{w^n\,|\, n \in \mathbb{N}\}}
  \end{align*}
  satisfy $V \cap W =\{0\}$.
  On the other hand we have
  \begin{align*}
    \frac{(v^n,w^n)}{\|v^n\|\|w^n\|} = \sqrt{\frac{n^2}{n^2+1}} \underset{n\to \infty}{\longrightarrow} 1
  \end{align*}
  and thus $\alpha(V,W) = 1$ and $\Angle(V,W) = 0$ which gives a counter example for Lemma~\ref{lemma:angle}.
  Now consider the orthogonal projection $P:H \to W$.
  Then we have
  \begin{align*}
    (Px)_k
    &=
    \begin{cases}
      x_k &\text{ if $k$ is even},\\
      0 &\text{ if $k$ is odd},
    \end{cases}
    &
    ((I-P)x)_k
    &=
    \begin{cases}
      0 &\text{ if $k$ is even},\\
      x_k &\text{ if $k$ is odd}
    \end{cases}
  \end{align*}
  and
  \begin{align*}
    \frac{\|v^n\|}{\|(I-P)v_n\|}
      = \sqrt{n^2+1} \underset{n\to \infty}{\longrightarrow} \infty.
  \end{align*}
  Hence there is no finite constant $C$ such that the inverse estimate $\|v\| \leq C \|(I-P)v\|$
  holds for all $v \in V$
  which gives a counter example for Lemma~\ref{lemma:inverse_ineq}.
  Finally, consider the bilinear form $a(v,w) = ((I-P)v,(I-P)w)$.
  Then we have $\operatorname{ker} a = W$
  and $a(\cdot,\cdot)$ is coercive on $(\operatorname{ker} a)^\perp$
  because $a(v,v) = \|v\|^2$ for $v \in (\operatorname{ker} a)^\perp = (I-P)(H)$.
  However, for $v^n \in V$ we have
  \begin{align*}
    \frac{a(v^n,v^n)}{\|v^n\|^2} = \frac{1}{n^2+1}  \underset{n\to \infty}{\longrightarrow} 0
  \end{align*}
  such that $a(\cdot,\cdot)$ is not coercive on $V$
  which gives a counter example for Proposition~\ref{prop:coercivity}.
\end{remark}

Although orthogonal projections and angles are at the heart of the above given statements and proofs,
we also give a related result in a Banach space setting, that is even true without the
assumption of finite dimensionality. Notice that this result is very close to
(but not a consequence of) a classical result based on projections as shown e.g.
in \cite[Lemma~4.1.3]{Ziemer1989}.

\begin{proposition}\label{prop:coercivity_banachspace}
    Let $X$ be a Banach space that is compactly embedded into a normed space $Y$
    and denote the corresponding norms by $\|\cdot\|_X$ and $\|\cdot\|_Y$, respectively.
    Furthermore let $|\cdot|$ be a continuous semi-norm on $X$
    such that $\|\cdot\|_X \leq C(\|\cdot\|_Y + |\cdot|)$ for some $C>0$
    and $V$ a closed subspace of $X$
    with $V\cap \{v\in X \,|\, |x|=0\}= \{0\}$.
    Then there is a constant $\gamma>0$ such that
    \begin{align*}
      \gamma\|v\|_Y \leq |v| \qquad \forall v \in V
    \end{align*}
    and thus $|\cdot|$ is coercive on $V$ with
    \begin{align*}
      \frac{\gamma}{C(1+\gamma)}\|v\|_X \leq |v| \qquad \forall v \in V.
    \end{align*}
\end{proposition}
\begin{proof}
  Assume that there is no such $\gamma>0$. Then there is a sequence $v_n \in V$ such that
  \begin{align*}
    \|v_n\|_Y = 1, \qquad |v_n| < \frac{1}{n}
  \end{align*}
  and thus $\|v_n\|_X \leq C2$ for all  $n \in \mathbb{N}$.
  By the compact embedding there is a subsequence (wlog. also denoted by $v_n$)
  and some $v \in Y$ such that $\|v_n -v\|_Y \to 0$.
  Using this, $|v_n| \to 0$, and
  \begin{align*}
    \|v_k - v_l\|_X \leq C\Bigl(\|v_k - v\|_Y + \|v_l - v\|_Y + |v_k| + |v_l| \Bigr)
  \end{align*}
  we find that $v_n$ is a $\|\cdot\|_X$-Cauchy-sequence
  and thus $\|v_n - \tilde{v} \|_X \to 0$ for some $\tilde{v} \in V$.
  By the compact embedding, $\|\cdot\|_Y$ is $\|\cdot\|_X$-continuous
  such that we also have $\|v_n - \tilde{v}\|_Y \to 0$ and thus $v = \tilde{v} \in V$.
  Frome those convergence results we get for $v$
  \begin{align*}
    \|v_n - v\|_Y &\to 0
      & &\Rightarrow &
      &1=\|v_n\|_Y \to \|v\|_Y=1,\\
    \|v_n - v\|_X &\to 0
      & &\Rightarrow &
      &|v_n| \to |v|=0
  \end{align*}
  and hence $0 \neq v \in V \cap \{v\in X \,|\, |x|=0\}$
  which contradicts the assumptions.
\end{proof}

\section{Poincar\'e-type inequalities in $H^m(\Omega)$}
\label{sec:poincare}
Now we consider Poincar\'e-type inequalities in $H^m(\Omega)$ with $m \in \N_0$.
Throughout this section let $\Omega \subset \R^d$ be a bounded domain with
Lipschitz boundary. On $H^m(\Omega)$ we use the inner product
\begin{align*}
    (u,v)_m = \sum_{|s|\leq m} \int_\Omega D^s u D^s v \, dx
\end{align*}
and the induced norm $\|\wc\|_m$,
where we used the classical multi-index notation.

In the following we investigate coercivity of the bilinear form
\begin{align*}
    \Prod{u,v}_m = \sum_{|s|= m} \int_\Omega D^s u D^s v \, dx
\end{align*}
and augmented variants on subspaces of $H^m(\Omega)$.
Note that this bilinear form induces the $H^m$-seminorm $|\wc|_m = \Prod{ \wc,\wc }_m^{1/2}$.
The main ingredients are the characterization
of the kernel of $\Prod{ \wc,\wc }_m$ and
the coercivity on its orthogonal complement.

\begin{lemma}\label{lem:semi_norm_kernel}
    The kernel of $\Prod{ \wc,\wc }_m$ is given by $\op{ker} (\Prod{ \wc,\wc }_m) = \mathcal{P}_{m-1}$
    where $\mathcal{P}_k$ is the space of polynomials with degree $\leq k$
    on $\R^d$.
\end{lemma}
\begin{proof}
    Let $\Prod{ v,v }_m = 0$. Then we have $D^s v=0$ for all multi-indices
    $s$ with $|s|=m$ and hence $v \in \mathcal{P}_{m-1}$.
\end{proof}

Next we show that $\Prod{ \wc,\wc }_m$ is coercive on the orthogonal
complement of $\mathcal{P}_{m-1}$. To this end we need the
following classical version of the Poincar\'e inequality on $H^m(\Omega)$.
All other versions will be derived from this one.

\begin{theorem}\label{thm:poincare_Hm}
    There is a constant $C>0$ such that
    \begin{align*}
        \|v\|_m^2 \leq C \Bigl(
            |v|_m^2
            + \sum_{|s|<m} \bigl( \int_\Omega D^s v\, dx\bigr)^2
            \Bigr)
            \qquad \forall v \in H^m(\Omega).
    \end{align*}
\end{theorem}
\begin{proof}
    See \cite[Theorem~7.2]{Wloka1987}.
\end{proof}

\begin{lemma}\label{lem:projected_poincare}
    Let $P:H^m(\Omega) \to \mathcal{P}_{m-1}$
    be the orthogonal projection into $\mathcal{P}_{m-1}$.
    Then
    \begin{align*}
        \|v-P v\|_m^2 \leq C
            |v|_m^2
            \qquad \forall v \in H^m(\Omega)
    \end{align*}
    for the same constant $C$ as in Theorem~\ref{thm:poincare_Hm}.
\end{lemma}
\begin{proof}
    We will use the modified inner product
    \begin{align*}
        \tripleprod{u,v}_m
            = \sum_{|s|= m} \int_\Omega D^s u D^s v \, dx
            + \sum_{|s|<m} \int_\Omega D^s u\, dx \int_\Omega D^s v\, dx
    \end{align*}
    on $H^m(\Omega)$. By Theorem~\ref{thm:poincare_Hm} this
    induces an equivalent norm $\triplenorm{\wc}_m$. The orthogonal
    projection into $\mathcal{P}_{m-1}$ with respect to $\triplenorm{\wc}_m$
    will be denoted by $\hat{P}$.

    Now let $u \in H^m$. Utilizing $D^s v = 0$ for $|s|=m$ and any $v \in \mathcal{P}_{m-1}$
    and Galerkin-orthogonality we get
    \begin{align}\label{eq:galerkin}
        0 = \tripleprod{u-\hat{P}u,v}_m
        = \sum_{|s|<m} \int_\Omega D^s (u-\hat{P}u)\, dx \int_\Omega D^s v\, dx \qquad \forall v \in \mathcal{P}_{m-1}.
    \end{align}
    We will inductively show that \eqref{eq:galerkin} induces
    \begin{align}\label{eq:int_zero}
        \int_\Omega D^s (u-\hat{P}u)\, dx = 0
    \end{align}
    for all $|s|<m$.
    For $s=(0,\dots,0)$ this follows from testing with $v= x^s = 1 = D^s v$.
    Now let $|s'|<m$ and assume that \eqref{eq:int_zero} is true for all $|s|<|s'|$.
    Then all terms in \eqref{eq:galerkin} for such $s$ vanish.
    To handle the remaining terms set $v=x^{s'}$.
    Then we have $D^r v = 0$ if either $|r|>|s'|$ or both $|r|=|s'|$ and $r\neq s'$ are true
    and $D^{s'} v = \operatorname{const} \neq 0$.
    Hence testing with $v$ gives \eqref{eq:int_zero} with $s=s'$.

    Using the best approximation property of $Pu$ with respect to $\|\cdot\|_m$,
    Theorem~\ref{thm:poincare_Hm}, identity \eqref{eq:int_zero},
    and $D^s \hat{P}v = 0$ for $|s|=m$ we get
    \begin{align}\label{eq:projection_estimate}
        \|u-P u\|_m^2
            &\leq \|u-\hat{P} u\|_m^2
            \leq C \triplenorm{u-\hat{P} u}_m^2
            = C |u-\hat{P} u|_m^2
            = C |u|_m^2.
    \end{align}
\end{proof}

As an immediate consequence of this and $(\mathcal{P}_{m-1})^\perp = (I-P)(H^m(\Omega))$ we get:
\begin{corollary}
    \label{cor:semi_norm_coercivity}
    The bilinear form $\Prod{\wc,\wc}_m$ is coercive on $(\op{ker} (\Prod{\wc,\wc}_m))^\perp = (\mathcal{P}_{m-1})^\perp$.
\end{corollary} 

As a consequence of the kernel characterization in Lemma~\ref{lem:semi_norm_kernel}
and the coercivity result in Corollary~\ref{cor:semi_norm_coercivity}
we can use Proposition~\ref{prop:coercivity}
to show coercivity on subspaces of $H^m(\Omega)$.

\begin{corollary}\label{cor:coercive_Hm_subspace}
    Let $V\subset H^m(\Omega)$ be a closed subspace with $V \cap \mathcal{P}_{m-1}= \{0\}$.
    Then $\Prod{ \wc,\wc }_m$ is coercive on $V$ and $|\wc|_m$ is equivalent to $\|\wc\|_m$ on $V$.
\end{corollary}

Now we will show some examples of Poincar\'e- or Friedrichs-type inequalities
or related coercivity results.

\begin{example}\label{ex:projected_poincare_H1}
    There is a constant $C_p$ such that
    \begin{align*}
        \|\textstyle v - \tfrac{1}{|\Omega|} \int_\Omega v\, dx \|_1^2
        \leq C_p |v|_1^2 \qquad \forall v \in H^1(\Omega).
    \end{align*}
\end{example} 
\begin{proof}
    Since $v \mapsto \tfrac{1}{|\Omega|} \int_\Omega v \, dx \in \mathcal{P}_0$
    is an orthogonal projection this is a special case of Lemma~\ref{lem:projected_poincare}.
\end{proof} 

\begin{example}
    Let $\Gamma \subset \partial \Omega$ with
    nonzero $(n-1)$-dimensional Hausdorff measure.
    Then $\Prod{\wc,\wc}_1$ is coercive on
    $H^1_{\Gamma,\int,0}(\Omega) = \{v \in H^1(\Omega) \st \int_\Gamma v \, ds =0 \}$.
\end{example} 
\begin{proof}
    By the trace theorem $v \mapsto v|_\Gamma \subset L^2(\Gamma)$
    is a continuous map and hence $H^1_{\Gamma,\int,0}(\Omega) \subset H^1(\Omega)$ is a closed subspace.
    Since $H^1_{\Gamma,\int,0}(\Omega) \cap \mathcal{P}_0 = \{0\}$
    Corollary~\ref{cor:coercive_Hm_subspace} provides the assertion.
\end{proof} 

\begin{example}\label{ex:poincare_boundary_value}
    Let $\Gamma \subset \partial \Omega$ with
    nonzero $(n-1)$-dimensional Hausdorff measure.
    Then $\Prod{\wc,\wc}_1$ is coercive on
    $H^1_{\Gamma,0}(\Omega) = \{v \in H^1(\Omega) \st v|_\Gamma =0 \}$.
\end{example}
\begin{proof}
    We only need to note that 
    $H^1_{\Gamma,0} (\Omega)$ is a closed subspace of $H^1_{\Gamma,\int,0}(\Omega)$.
\end{proof} 

\begin{example}
    Let $\Gamma \subset \partial \Omega$ with
    nonzero $(n-1)$-dimensional Hausdorff measure.
    Then
    \begin{align*}
        \|v\|_1^2 \leq C |v|_1^2 + \frac{|\Omega|}{|\Gamma|^2}\bigl(\textstyle\int_{\Gamma} v \,ds \bigr)^2
        \qquad \forall v \in H^1(\Omega)
    \end{align*}
    where $C$ is the coercivity constant from Example~\ref{ex:projected_poincare_H1}.
\end{example}
\begin{proof}
    Let $v \in H^1(\Omega)$ then Example~\ref{ex:projected_poincare_H1} provides
    \begin{align*}
        \|v\|_1^2
            \leq \|v-P v\|_1^2 + \|Pv\|_1^2
            \leq C|v|_1^2 + \|Pv\|_0^2
            = C |v|_1^2
                + \frac{|\Omega|}{|\Gamma|^2}\bigl(\textstyle\int_{\Gamma} v \,ds \bigr)^2.
    \end{align*}
\end{proof} 

As a direct consequence we get a version of Friedrichs' inequality
with boundary integrals.
\begin{example}
    Let $\Gamma \subset \partial \Omega$ with
    nonzero $(n-1)$-dimensional Hausdorff measure.
    Then there is a constant $C$ with
    \begin{align*}
        \|v\|_1^2 \leq C\Bigl( |v|_1^2 + \|v\|_{L^2(\Gamma)}^2 \Bigr)
        \qquad \forall v \in H^1(\Omega).
    \end{align*}
\end{example}

\begin{example}
    Let $d=1,2,3$ and $p_1,\dots,p_{d+1} \subset \overline{\Omega}$ affine independent.
    Then $\Prod{\wc,\wc}_2$ is coercive on
    $V = \{v \in H^2(\Omega) \st v(p_1)=\dots=v(p_{d+1})=0\}$.
\end{example} 
\begin{proof}
    By the Sobolev embedding $V$ is closed.
    Furthermore $V\cap \mathcal{P}_1 = \{0\}$ and Corollary~\ref{cor:coercive_Hm_subspace}
    provides the assertion.
\end{proof} 

\begin{example}\label{ex:poincare_navier_bc}
    The bilinear form $\Prod{\wc,\wc}_2$ is coercive on $H^2(\Omega)\cap H^1_0(\Omega)$.
\end{example} 
\begin{proof}
    Since $\Omega$ is bounded we have $H^1_0(\Omega)\cap \mathcal{P}_1 = \{0\}$.
    Hence Corollary~\ref{cor:coercive_Hm_subspace} provides the assertion.
\end{proof} 

\begin{example}
    Let $d=1,2,3$ and $p_1,\dots,p_{d+1} \subset \overline{\Omega}$ affine independent.
    Then the bilinear form $\Prod{\wc,\wc}_2 + b(\wc,\wc)$ with
    \begin{align*}
        b(u,v) = \sum_{i=1}^{d+1} u(p_i)v(p_i)
    \end{align*} 
    is coercive on $H^2(\Omega)$.
\end{example} 
\begin{proof}
    Symmetry and positive semi-definiteness of $b(\wc,\wc)$ are obvious.
    Positive definiteness on $\mathcal{P}_1$ follows from affine independence.
    Finally, the Sobolev embedding implies continuity such that
    Proposition~\ref{prop:coercivity_augmented} provides the assertion.
\end{proof}

\section{Coercivity of the bi- and quadruple-Laplacian operator}
\label{sec:biharmonic}
In the following we show coercivity of the operators $\Delta^2$
and $\Delta^4$ with various boundary conditions.
Since such operators often arise in the context of plate-like
problems, we restrict our considerations to piecewise smooth
domains $\Omega \subset \R^2$. In the following $\nu$ and $\tau$
will denote piecewise smooth oriented unit normal and tangential fields.

We are especially interested in periodic boundary conditions.
To this end we define for the special case of a rectangle
$\Omega$ the periodic spaces
\begin{align*}
    C_{p}^\infty(\Omega) & = \{v|_\Omega  \st v \in C^\infty(\R^2) \text{ is $\Omega$-periodic}\}, &
    H_{p}^k(\Omega) &= \overline{C_{p}^\infty(\Omega)}^{\|\wc\|_k}.
\end{align*} 

\begin{lemma}\label{lemma:bi_laplace}
    Let $V$ be any of the spaces
    \begin{itemize}
        \item
            $H^2_0(\Omega)$,
        \item
            $H^2_{p}(\Omega)$
            with rectangular $\Omega$,
    \end{itemize}
    then $\|\Delta v\|_0^2 = |v|_2^2$ for all $v \in V$.
\end{lemma}
\begin{proof}
    Let $\tilde{V}$
    a corresponding dense subspace of smooth functions given by
    $C_0^\infty(\overline{\Omega})$
    or $C_{p}^\infty(\Omega)$,
    respectively.
    Then partial integration for $v \in \tilde{V}$ gives
    \begin{align}\label{eq:laplace_square}
        \begin{split}
            \| \Delta v\|_0^2
            &= |v|_2^2
                + \sum_{i,j=1}^2 \int_{\partial \Omega} \partial_{i} v \partial_{jj} v \nu_i - \partial_{ji} v \partial_{i}v \nu_j\, ds \\
            &= |v|_2^2
                + \int_{\partial \Omega} \tfrac{\partial v}{\partial \nu} \Delta v - \nabla \tfrac{\partial}{\partial \nu} v \cdot \nabla v \, ds \\
            &= |v|_2^2
                + \int_{\partial \Omega} \tfrac{\partial v}{\partial \nu} \tfrac{\partial^2}{\partial \tau^2} v - \tfrac{\partial}{\partial \tau} \tfrac{\partial}{\partial \nu} v \tfrac{\partial}{\partial \tau} v \, ds.
        \end{split} 
    \end{align} 

    For the case $\tilde{V}=C_0^\infty(\overline{\Omega})$ the boundary term obviously vanishes.
    For $V=H_{p}^2(\Omega)$ we can split the boundary according to
    $\partial \Omega = \Gamma_W \cup \Gamma_E \cup \Gamma_N \cup \Gamma_S$
    such that we have (up to translation)
    \begin{align*}
        v|_{\Gamma_W} &= v|_{\Gamma_E}, &
        v|_{\Gamma_N} &= v|_{\Gamma_S}, &
        \tfrac{\partial}{\partial \nu} v |_{\Gamma_W} &= -\tfrac{\partial}{\partial \nu} v|_{\Gamma_E}, &
        \tfrac{\partial}{\partial \nu} v |_{\Gamma_N} &= -\tfrac{\partial}{\partial \nu} v|_{\Gamma_S}.
    \end{align*}
    Now the minus sign (resulting from the flipped orientation of the normal) implies that
    boundary integrals from opposing boundary segments cancel out.
\end{proof} 

\begin{proposition}\label{prop:H2_BC}
    Let $V$ be any of the spaces
    \begin{itemize}
        \item
            $H^2_0(\Omega)$,
        \item
            $\{v \in H^2_{p}(\Omega) \st \Textint_{\partial \Omega} v \, ds=0\}$
            with rectangular $\Omega$,
        \item
            $\{v \in H^2_{p}(\Omega) \st \Textint_{\Omega} v \, dx=0\}$
            with rectangular $\Omega$,
    \end{itemize}
    then the bilinear form
    $a(u,v) = \Textint_\Omega \Delta u \Delta v \, dx$
    is coercive on $V$.
\end{proposition} 
\begin{proof}
    By Lemma~\ref{lemma:bi_laplace} we have $a(v,v) = |v|_2^2$ for $v \in V$.
    Since $\mathcal{P}_1 \cap V = \{0\}$
    for any choice of $V$, Corollary~\ref{cor:coercive_Hm_subspace}
    now provides the assertion.
\end{proof} 

\begin{remark}
  The bilinear form
  $a(u,v) = \Textint_\Omega \Delta u \Delta v \, dx$
  is also coercive on the space
  $H^2(\Omega) \cap H^1_0(\Omega)$ for convex or $C^2$-regular domains (see e.g. \cite{NazarovSweers2007}).
  This follows from classic elliptic regularity theory which provides
  that for $f=\Delta v \in L^2(\Omega)$ the solution $v$ of the second order
  equation $\Delta v = f$
  is $H^2$-regular and satisfies the bound
  \begin{align*}
    \|v\|_2^2 \leq C \|f\|_0^2 = C \|\Delta v\|_0^2.
  \end{align*}
\end{remark}

For the quadruple-Laplacian we get similar results:
\begin{lemma}\label{lemma:quad_laplace}
    Let $V$ be any of the spaces
    \begin{itemize}
        \item
            $H^4_0(\Omega)$,
        \item
            $H^4_{p}(\Omega)$
            with rectangular $\Omega$,
    \end{itemize}
    then $\|\Delta^2 v\|_0^2 = |v|_4^2$ for all $v \in V$.
\end{lemma}

\begin{proof}
    Let $\tilde{V}$ a corresponding dense subspace of smooth functions given by
    $C_0^\infty(\overline{\Omega})$ or $C_{p}^\infty(\Omega)$,
    respectively.
    Now let $v \in \tilde{V}$ and $w=\Delta v \in H^2(\Omega)$.
    Then we can apply Lemma~\ref{lemma:bi_laplace} to get
    \begin{align*}
        \|\Delta^2 v\|_0^2
        = |\Delta v|_2^2
        = \sum_{|s|=2} \|D^s \Delta v\|_0^2
        = \sum_{|s|=2} \|\Delta D^s v\|_0^2.
    \end{align*} 
    Hence it remains to show $\| \Delta z\|_0^2 = |z|_2^2$
    for $z = D^sv$ and $|s|=2$.
    To this end we again apply partial integration as in~\eqref{eq:laplace_square}
    to get
    \begin{align}
        \|\Delta z\|_0^2
        &= |z|_2^2
            + \int_{\partial \Omega} \tfrac{\partial z}{\partial \nu} \tfrac{\partial^2}{\partial \tau^2} z - \tfrac{\partial}{\partial \tau} \tfrac{\partial}{\partial \nu} z \tfrac{\partial}{\partial \tau} z \, ds.
    \end{align} 
    In the case of $V=H^4_0(\Omega)$ all boundary terms vanish.
    Similarly, for $V=H^4_{p}(\Omega)$ periodicity of $v$ implies periodicity
    of $z$, such that the boundary term vanishes by the same arguments as in Lemma~\ref{lemma:bi_laplace}.
\end{proof}

\begin{proposition}\label{prop:H4_BC}
    Let $V$ be any of the spaces
    \begin{itemize}
        \item
            $H^4_0(\Omega)$,
        \item
            $\{v \in H^4_{p}(\Omega) \st \Textint_{\partial \Omega} v \, ds=0\}$
            with rectangular $\Omega$,
        \item
            $\{v \in H^4_{p}(\Omega) \st \Textint_{\Omega} v \, dx=0\}$
            with rectangular $\Omega$,
    \end{itemize}
    then the bilinear form
    $a(u,v) = \Textint_\Omega \Delta^2 u \Delta^2 v \, dx$
    is coercive on $V$.
\end{proposition} 
\begin{proof}
    By Lemma~\ref{lemma:quad_laplace} we have $a(v,v) = |v|_4^2$ for $v \in V$.
    In view of Corollary~\ref{cor:coercive_Hm_subspace} it remains to show
    $\mathcal{P}_3 \cap V = \{0\}$ for all choices of $V$. To this end let
    $p \in \mathcal{P}_3 \cap V$.

    For $V = H^4_0(\Omega) \subset H^4(\Omega) \cap H^3_0(\Omega)$ the
    boundary conditions 
    provide $D^s p =0$ for $|s|\leq 2$ on $\partial \Omega$.
    For $|s|=2$ we have $D^s p \in \mathcal{P}_1$ and thus $D^s p = 0$ on $\Omega$.
    Hence $p$ is bilinear and we get for $|s|=1$ that $D^s p \in \mathcal{P}_1$ and thus $D^s p = 0$.
    As a consequence $p$ is constant which gives $p=0$.
    
    Finally periodicity of $p \in H_{p}^4(\Omega)$ implies that $p=\op{const}$
    which together with $\Textint_{\partial \Omega} p \, ds =0$ or $\Textint_\Omega p \, dx = 0$
    gives $p=0$.
\end{proof} 

\begin{remark}
  Similar to the $H^2$-case above,
  coercivity of the bilinear form
  $a(u,v) = \Textint_\Omega \Delta^2 u \Delta^2 v \, dx$
  on the space
  $H^4(\Omega) \cap H^3_0(\Omega)$
  for convex or $C^2$-regular domains
  follows from elliptic regularity.
  To see this we note that $v \in H^4(\Omega) \cap H^3_0(\Omega)$
  implies that $D^sv=0$ on $\partial \Omega$ for $|s|\leq 2$.
  Hence we can apply the same argument as  in the $H^2$-case several times to get
  \begin{align*}
    C^2\|\Delta^2 v\|_0^2
      &\geq C \|\Delta v\|_2^2
      = C \sum_{|s|\leq 2} \|\Delta D^s v\|_0^2
      \geq \sum_{|s|\leq 2} \|D^s v\|_2^2
      \geq \sum_{|s|\leq 4} \|v\|_0^2 = \|v\|_4^2.
  \end{align*}
\end{remark}

\begin{remark}
  For the space
  $H^4_\Delta(\Omega)=\{v \in H^4(\Omega)\st v=0,\; \Delta v =0 \text{ on }\partial \Omega\}$
  with rectangular $\Omega$
  coercivity of the bilinear form
  $a(u,v) = \Textint_\Omega \Delta^2 u \Delta^2 v \, dx$
  can be shown by combining elliptic regularity with Corollary~\ref{cor:coercive_Hm_subspace}.
  Without loss of generality assume that $\Omega$ is parallel to the axes
  such that partial derivatives correspond to tangential and normal derivatives.
  First note that $v \in H^4_\Delta(\Omega)$ implies
  \begin{align}\label{eq:delta_bc}
    \tfrac{\partial}{\partial \tau} v  = \tfrac{\partial^2}{\partial \tau^2}v = 0, \qquad
    \tfrac{\partial^2}{\partial \nu^2}v = \Delta v - \tfrac{\partial^2}{\partial \tau^2}v = 0.
  \end{align}
  Due to $\Delta v=0$ on $\partial \Omega$ we can use elliptic regularity once to get
  \begin{align*}
    C \|\Delta^2 v\|_0^2
      &\geq \|\Delta v\|_2^2
      = \sum_{|s|\leq 2} \|\Delta D^s v\|_0^2
      \geq \sum_{|s|= 2} \|\Delta D^s v\|_0^2
  \end{align*}
  Now the boundary conditions \eqref{eq:delta_bc} allow to again use elliptic regularity for the
  non-mixed second order terms $D^s v$ with $s=(2,0)$ and $s=(0,2)$ to get
  \begin{align*}
    C \|\Delta D^s v\|_0^2 \geq \|D^s v\|_2^2 \geq |D^s v|_2^2.
  \end{align*}
  For the remaining mixed term $z=D^s v$ with $s=(1,1)$
  using partial integration as in the proof of Lemma~\ref{lemma:quad_laplace} gives
  \begin{align*}
    \|\Delta D^{s}v\|_0^2 = |D^s v|_2^2
  \end{align*}
  because all boundary terms vanish.
  Hence we have shown $\|\Delta^2v\|_0^2 \geq \rho |v|_2^2$
  for some constant $\rho>0$.
  Coercivity now follows from Corollary~\ref{cor:coercive_Hm_subspace}
  because $H^4_\Delta(\Omega)\cap \mathcal{P}_3 = \{0\}$.
  To see this let $p \in H^4_\Delta(\Omega)\cap \mathcal{P}_3$.
  Then the boundary conditions
  \eqref{eq:delta_bc} imply for $s=(2,0)$ and $s=(0,2)$ that $D^s p = 0$ on $\partial \Omega$
  which, together with $D^s p \in \mathcal{P}_1$ implies $D^s p = 0$ on $\Omega$.
  Hence $p$ is bilinear on the rectangle $\Omega$ which together with $p|_{\partial \Omega}=0$
  gives $p=0$.
\end{remark}

\bibliography{paper}
\bibliographystyle{plainurl}

\end{document}